\documentclass[11pt,reqno]{amsart}
\usepackage{amsmath}
\usepackage{amssymb}
\usepackage[left=3cm,top=3cm,right=3cm,bottom=3cm]{geometry}
\usepackage{epsfig}
\usepackage{todonotes}
 \usepackage{caption}
\usepackage{subcaption}

\begin{document}
\newtheorem{theorem}{Theorem}[section]
\newtheorem{lemma}[theorem]{Lemma}
\newtheorem{definition}[theorem]{Definition}
\newtheorem{conjecture}[theorem]{Conjecture}
\newtheorem{proposition}[theorem]{Proposition}
\newtheorem{algorithm}[theorem]{Algorithm}
\newtheorem{corollary}[theorem]{Corollary}
\newtheorem{observation}[theorem]{Observation}
\newtheorem{claim}[theorem]{Claim}
\newtheorem{problem}[theorem]{Open Problem}
\newtheorem{remark}[theorem]{Remark}
\newcommand{\noin}{\noindent}
\newcommand{\ind}{\indent}
\newcommand{\om}{\omega}
\newcommand{\I}{\mathcal I}
\newcommand{\pp}{\mathcal P}
\newcommand{\ppp}{\mathfrak P}
\newcommand{\N}{{\mathbb N}}
\newcommand{\LL}{\mathbb{L}}
\newcommand{\R}{{\mathbb R}}
\newcommand{\E}[1]{\mathbb{E}\left[#1 \right]}
\newcommand{\V}{\mathbb Var}
\newcommand{\Prob}{\mathbb{P}}
\newcommand{\eps}{\varepsilon}

\newcommand{\Tv}{P}

\newcommand{\mT}{\mathcal{T}}
\newcommand{\mS}{\mathcal{S}}
\newcommand{\mA}{\mathcal{A}}
\newcommand{\mB}{\mathcal{B}}

\newcommand{\Cyc}[1]{\mathrm{Cyc}\left(#1\right)}
\newcommand{\Seq}[1]{\mathrm{Seq}\left(#1\right)}
\newcommand{\Mul}[1]{\mathrm{Mul}\left(#1\right)}
\newcommand{\Mulo}[1]{\mathrm{Mul}_{>0}\left(#1\right)}
\newcommand{\Set}[1]{\mathrm{Set}\left(#1\right)}
\newcommand{\Setd}[1]{\mathrm{Set}_{d}\left(#1\right)}

\title{On the hyperbolicity of random graphs}

\author{Dieter Mitsche}
\address{Universit\'{e} de Nice Sophia-Antipolis, Laboratoire J-A Dieudonn\'{e}, Parc Valrose, 06108 Nice cedex 02}
\email{\texttt{dmitsche@unice.fr}}

\author{Pawe\l{} Pra\l{}at}
\address{Department of Mathematics, Ryerson University, Toronto, ON, Canada}
\email{\tt pralat@ryerson.ca}

\thanks{The second author gratefully acknowledges support from NSERC and Ryerson University.}


\begin{abstract}
Let $G=(V,E)$ be a connected graph with the usual (graph) distance metric $d:V \times V \to \N \cup \{0 \}$. Introduced by Gromov, $G$ is $\delta$-hyperbolic if for every four vertices $u,v,x,y \in V$, the two largest values of the three sums $d(u,v)+d(x,y), d(u,x)+d(v,y), d(u,y)+d(v,x)$ differ by at most $2\delta$.
In this paper, we determine precisely the value of this hyperbolicity for most binomial random graphs. 
\end{abstract}

\maketitle

\section{Introduction}\label{sec:intro}
Hyperbolicity is a property of metric spaces that generalizes the idea of negatively curved spaces like the classical hyperbolic space or Riemannian manifolds of negative sectional curvature (see, for example,~\cite{book1, book2}). Moreover, this concept can be applied to discrete structures such as trees and Cayley graphs of many finitely generated groups. The study of properties of Gromov's hyperbolic spaces from a theoretical point of view is a topic of recent and increasing interest in graph theory and computer science. Informally, in graph theory hyperbolicity measures how similar a given graph is to a tree---trees have hyperbolicity zero and graphs that are ``tree-like'' have ``small'' hyperbolicity. Formally, a connected graph $G=(V,E)$ is \textbf{$\delta$-hyperbolic}, if for every four vertices $u,v,x,y \in V$, the two largest values in the set 
$$
\{d(u,v)+d(x,y), d(u,x)+d(v,y), d(u,y)+d(v,x) \}
$$
differ by at most $2\delta$. The \textbf{hyperbolicity} of $G$, denoted by $\delta_H(G)$, is the smallest $\delta$ for which this property holds.

Our results below show a close relation between the diameter and hyperbolicity. This relation was also studied in~\cite{Benj2}: the authors show that for vertex transitive graphs, the hyperbolicity is within a constant factor of the diameter. The author of~\cite{Benj1} bounds the number of vertices in terms of the Cheeger constant and the hyperbolicity, showing that  the family of expanders is not uniformly $\delta$-hyperbolic for $\delta$ constant.
In~\cite{Bandelt1} several equivalent conditions for a graph to be $0$-hyperbolic are given, and in~\cite{Bandelt2} the authors characterize $1/2$-hyperbolic graphs in terms of forbidden subgraphs.  On the algorithmic side, by the conditions investigated in~\cite{Bandelt1}, $0$-hyperbolic graphs can be recognized in linear time, and in~\cite{Coudert2} it is shown that recognizing $1/2$-hyperbolic graphs is equivalent to finding an induced cycle of length $4$ in a graph. Fast algorithms for computing the hyperbolicity of large-scale graphs are given in~\cite{Coudert}. \\

The study of this parameter  is motivated by the following observations: on the algorithmic side, in~\cite{Chepoi2} fast algorithms for computing properties related to the diameter of graphs with small hyperbolicity are given. In~\cite{Chepoi3} the authors give a simple construction showing that distances in graphs with small hyperbolicity can be approximated within small error by corresponding trees.  In~\cite{Chepoi1a} the authors give a polynomial algorithm which computes for such graphs an augmented graph of at most a given diameter, and whose number of added edges is within a constant factor of the minimum number of added edges that are needed such that the augmented graph has at most such a diameter. Finally, in~\cite{Chepoi1} it is shown that all cop-win graphs in which the cop and the robber move at different speeds have small hyperbolicity, and also a constant-factor approximation of $\delta_H$ in time $O(n^2 \log \delta)$ is given. 
Moreover, the concept of hyperbolicity turns out to be useful for many applied problems such as visualization of the Internet, the Web graph, and other complex networks~\cite{Munzner1, Munzner2}, routing, navigation, and decentralized search in these networks~\cite{boguna, kleinberg}. In particular, hyperbolicity plays an important role when investigating the spread of viruses through a network~\cite{jonckheere}.

\bigskip

Let us recall a classic model of random graphs that we study in this paper. The \textbf{binomial random graph} $\mathcal{G}(n,p)$ is defined as a random graph with vertex set $[n]=\{1,2,\dots, n\}$ in which a pair of vertices appears as an edge with probability $p$, independently for each such a pair. As typical in random graph theory, we shall consider only asymptotic properties of $\mathcal{G}(n,p)$ as $n\rightarrow \infty$, where $p=p(n)$ may and usually does depend on $n$. We say that an event in a probability space holds \textbf{asymptotically almost surely} (\textbf{a.a.s.}) if its probability tends to one as $n$ goes to infinity. 

We say that $f(n) = O(g(n))$  if there exists an integer $n_0$ and a constant $c > 0$ such that $|f(n)| \leq c |g(n)|$ for all $n \geq n_0$, and $f(n) =\Omega(g(n))$, if $g(n) = O(f(n))$. Also, $f(n) = \omega(g(n))$, or $f(n) \gg g(n)$,  if $\lim_{n \to \infty} |f(n)|/|g(n)| = \infty$, and $f(n) = o(g(n))$ or $f(n) \ll g(n)$, if $g(n) = \omega(f(n))$.
Throughout this paper, $\log n$ always denotes the natural logarithm of $n$.
\bigskip 

In this paper, we investigate the hyperbolicity for binomial random graphs. Surprisingly, this important graph parameter is not well investigated for random graphs which is an important and active research area with numerous applications. In~\cite{sparse_graphs}, sparse random graphs ($p=p(n)=c/n$ for some real number $c>1$) are analyzed. It was shown that $\mathcal{G}(n,p)$ is, with positive probability, not $\delta$-hyperbolic for any positive $\delta$. Nothing seems to be known for $p \gg n^{-1}$. On the other hand, it is known that for a random $d$-regular graph $G$, for $d \ge 3$, we have that a.a.s.
$$
\frac 12 \log_{d-1} n - \omega(n) \le \delta_H(G) \le \frac 12 \log_{d-1} n + O(\log \log n),
$$
where $\omega(n)$ is any function tending to infinity together with $n$. (In fact, almost geodesic cycles are investigated in~\cite{d-reg}, and this is an easy consequence of this result.) The hyperbolicity of the class of Kleinberg's small-world random graphs is investigated in~\cite{small-world}. 

\bigskip
Our contribution is the following result.
\begin{theorem}~\label{thm:main}
Let $G \in \mathcal{G}(n,p)$. \\
Suppose first that 
$$
d=d(n)=p(n-1) \gg \frac {\log^5 n}{(\log \log n)^2} \textrm{ \ \ \ and \ \ \  } p = 1 - \omega(1/n^2).
$$  
Let $j \geq 2$ be the smallest integer such that $d^j/n - 2\log n \to \infty$.
Then, the following properties hold a.a.s. 
\begin{itemize}
\item[(i)] If $j$ is even and $d^{j-1} \leq \frac{1}{16} n \log n$, then $\delta_H(G) =j/2$.
\item[(ii)] If $j$ is even and $d^{j-1} > \frac{1}{16} n \log n$ (but still $d^{j-1} \leq (2+o(1)) n \log n$), then 
$$j/2-1 \le \delta_H(G) \le j/2.$$
\item[(iii)] If $j$ is odd, then $\delta_H(G)=(j-1)/2$.  
\end{itemize}
Furthermore, the following complementary results hold.
\begin{itemize}
\item[(iv)] For $p=1-2c/n^2$ for some constant $c > 0$, a.a.s.\ $\delta_H(G) \in \{0,1/2,1\}$. More precisely, 
\begin{eqnarray*}
\Prob{(\delta_H(G)=0)} &=& (1+o(1))e^{-c}, \\
\Prob{(\delta_H(G)=1/2)}&=&(1+o(1))ce^{-c}, \textrm{ and } \\
\Prob{(\delta_H(G)=1)}&=&(1+o(1))\left(1-(c+1)e^{-c}\right). 
\end{eqnarray*}
\item[(v)] For $p =1-o(1/n^2)$, a.a.s.\  $\delta_H(G)=0$.
\end{itemize}
\end{theorem}

\bigskip

\textbf{Remark.} It seems that with quite a bit more work, we could slightly push the lower bound required for $d$ and require only that $d \gg \log^3 n$ or perhaps even only $d \gg \log^2 n$. Unfortunately, it seems that it is more difficult to investigate sparser graphs (that is, assuming only $d \gg \log n$ or even closer to the connectivity threshold). Therefore, we aim for an easier (and cleaner) argument in this paper, leaving the investigation of sparser graphs as an open problem. Let us also mention that the hyperbolicity is not determined precisely for dense graphs right before the diameter decreases from even $j$ to $j-1$ (case (ii) in Theorem~\ref{thm:main}). Again, the constant $\frac {1}{16}$ could be slightly improved with a more delicate argument but the gap cannot be closed with the current approach. This is also worth investigating and (unfortunately) left open at the moment.
\section{Preliminaries} \label{sec:prel}

In this section, we introduce a few useful lemmas. The following result is well-known but we include the proof for completeness. 
\begin{lemma}\label{lem:upperbound}
Let $G$ be any connected graph with diameter at most $D$. Then $\delta_H(G) \leq D/2$.
\end{lemma}
\begin{proof}
Consider any four vertices $u,v,x,y$ with their three sums of distances $d_1=d(u,v)+d(x,y) \geq d_2=d(u,x)+d(v,y) \geq d_3=d(u,y)+d(v,x)$. We need to show that $d_1-d_2 \leq D$. Clearly, $d_1 \leq 2D.$ First observe that by applying the triangle inequality four times,
\begin{eqnarray*}
2d_1 &=& 2\left( d(u,v) + d(x,y) \right) \\
& \leq& d(u,y)+d(y,v) + d(u,x)+d(x,v)+d(x,u)+ d(u,y)+d(x,v)+d(v,y) \\
&=&2(d_2+d_3),
\end{eqnarray*}
and thus $d_1 \leq d_2+d_3$ or equivalently $d_1-d_2 \leq d_3$. Hence, if $d_3 \leq D$, the required condition holds and we are done. Otherwise, $d_3 > D$ and so also $d_2 > D$. As a consequence, $d_1 - d_2 < 2D - D=D$, and we are done as well.
\end{proof}
We can slightly improve the upper bound for graphs with odd diameter.
\begin{lemma}\label{lem:upperbound2}
Let $G$ be any graph with diameter at most $D=2k+1$ for some integer $k$. Then $\delta_H(G) \leq (D-1)/2=k.$
\end{lemma}
\begin{proof}
As in the previous proof, we consider four vertices $u,v,x,y$ with their three sums of distances $d_1=d(u,v)+d(x,y) \geq d_2=d(u,x)+d(v,y) \geq d_3=d(u,y)+d(v,x)$. Our goal is to show that $d_1-d_2 \leq D-1$. 
Arguing as in the previous proof, we get that $d_1 - d_2 \le d_3$. Hence, if $d_3 < D$, then $d_1 - d_2 \le D-1$ and we are done. Also, if $d_2 \geq D+1$, then $d_1 - d_2 \leq 2D - (D+1) \leq D-1$. So the only case to analyze is when $d_2=d_3=D$. For a contradiction, suppose that $d_1-d_2=D$, that is,
 $d_1=d(u,v)+d(x,y)=2D$. In particular, $d(u,v)=D$ and $d(x,y)=D$. Since $D=d_3=d(u,y)+d(v,x)$ and $D$ is odd, we may assume (without loss of generality) that $d(u,y) < D/2$. Then, since 
$$
D=d(u,v) \leq d(u,y)+d(y,v)
$$
and
$$
D=d(x,y) \leq d(x,u)+d(u,y),
$$
we have $d(y,v) > D/2$,  $d(x,u) > D/2$, and we have that $d_2 =d(u,x)+d(v,y) > D$, contradicting our assumption on $d_2$. Therefore, $d_1-d_2 \leq D-1$, and the lemma follows.
\end{proof}

\bigskip

In order to bound the hyperbolicity from above, we will make use of the following result for random graphs,  see~\cite[Corollary~10.12]{bol}.

\begin{lemma}[\cite{bol}, Corollary 10.12]\label{lem:diameter}
Suppose that $d = p(n-1) \gg \log n$ and 
\[
d^i/n - 2 \log n \to \infty \text{ \ \ \ \ and \ \ \ \ } d^{i-1}/n - 2 \log n \to -\infty.
\]
Then the diameter of  $G \in \mathcal{G}(n,p)$ is equal to $i$ a.a.s.
\end{lemma}
From the proof of this result, we have the following corollary.
\begin{corollary}\label{cor:diameter}
Suppose that $d = p(n-1) \gg \log n$ and that
 $$d^i/n - 2 \log n \to \infty.$$ Then the diameter of $G \in \mathcal{G}(n,p)$ is at most $i$ a.a.s. 
\end{corollary}

\bigskip

In order to obtain a lower bound on the hyperbolicity, we will need the following expansion lemma investigating the shape of typical neighbourhoods of vertices. Before we state the lemma we need a few definitions. For any $j \geq 0$, let us denote by $N(v,j)$ the set of vertices at distance at most $j$ from $v$, and by $S(v,j)$ the set of vertices at distance exactly $j$ from $v$. 
Also, for a set of vertices $F \subseteq V$, and $x \in V \setminus F$, denote by $N_{V \setminus F}(x,j)$ the set of vertices in $V \setminus F$ at distance at most $j$ from $x$ in the graph induced by $V \setminus F$, and similarly let $S_{V \setminus F}(x,j)$ be the set of vertices in $V \setminus F$ at distance exactly $j$ from $x$ in the graph induced by $V \setminus F$. 

\begin{lemma}\label{lem:expansion}
Suppose that $d=p(n-1)$ is such that 
$$
\frac {\log^5 n}{(\log \log n)^2} \ll d \le \left(\frac{1}{16}n\log n\right)^{1/3}.
$$
Let $G = (V,E) \in \mathcal{G}(n,p)$ and let $i \geq 4$ be the largest even integer such that $d^{i-1} \leq \frac{1}{16} n \log n.$ Fix any set $F \subseteq V$ such that $|F| = O(d^{i/2-1}) = O(\sqrt{n \log n/d})$ and fix any vertex $v \in V \setminus F$. Then, 
\begin{itemize}
\item [(i)] with probability $1-o(n^{-1})$ we have 
$$
|N_{V \setminus F}(v,1)| = d \left( 1 + o \left( \frac {\log \log n}{\log^2 n} \right) \right), 
$$
\item [(ii)] with probability 
$$
1-O(d^{i/2}/n) = 1-O(\sqrt{d \log n / n})=1-O((\log n)^{2/3} n^{-1/3})
$$ 
there is no edge from $v$ to $F$.
\end{itemize}

\bigskip \noindent In particular, it follows that a.a.s.\ the following properties hold:
\begin{itemize}
\item[(iii)] for all $j=1,2,\ldots,i/2 - 1$,
$$
|N_{V \setminus F}(v,j)| = |S_{V \setminus F}(v,j)| (1+O(1/d)) = d^j(1+o(\log^{-1} n)),
$$ 
\item [(iv)] for all $j=1,2,\ldots,i/2 - 2$, every vertex of $S_{V \setminus F}(v,j)$ has $d(1+o(\log \log n / \log^{2} n))$ neighbours in $S_{V \setminus F}(v, j+1)$,  
\item [(v)] the graph induced by $N_{V \setminus F}(v, i/2 -2)$ forms a tree,
\item [(vi)] all vertices of $S_{V \setminus F}(v,i/2-1)$ have exactly one neighbour in $S_{V \setminus F}(v,i/2-2)$,
\item [(vii)] for any fixed partition of the neighbours of $v$ into two sets, $V_L$ and $V_R$, such that $||V_L| - |V_R|| \le 1$, let $S_L$ denote the set of vertices of $S(v,i/2-1)$ that are at distance $i/2-2$ from $V_L$, and let $S_R=S(v,i/2-1) \setminus S_L$; then,
$$
|S_L| = |S_R| (1+o(\log^{-1} n))  = \frac {d^{i/2-1}}{2} (1+o(\log^{-1} n)).
$$
\end{itemize}
\end{lemma}
\begin{proof}
Let $F \subseteq V$, $f=|F| = O(d^{i/2-1})$, and $v \in V \setminus F$. Consider the random variable $X = X(F,v) = |S_{V \setminus F}(v,1)|$.  We will bound $X$ in a stochastic sense. There are two things that need to be estimated: the expected value of $X$, and the concentration of $X$ around its expectation. Since $X \in \textrm{Bin}(n-f-1,p)$, it is clear that
$$
\E{X} = \frac {d}{n-1} (n-f-1) = d \left( 1 + O(f/n) \right) = d (1+O(n^{-1/2})).
$$
A consequence of Chernoff's bound (see e.g.~\cite[Corollary~2.3]{JLR}) is that 
\begin{equation}\label{eq:chern}
\Prob \Big( |X-\E X| \ge \eps \E X \Big) \le 2\exp \left( - \frac {\eps^2 \E X}{3} \right)  
\end{equation}
for  $0 < \eps < 3/2$. Hence, after taking $\eps = 2 \sqrt{\log n / d}$, we get that with probability $1+o(n^{-1})$ we have 
$$
X = \E{X} (1+O(\eps)) = d (1+O(n^{-1/2})) (1+O(\eps)) = d \left( 1 + o \left( \frac {\log \log n}{\log^2 n} \right) \right).
$$
This proves part (i) of the lemma. 
Part (ii) is straightforward since the probability that there is no edge from $v$ to $F$ is equal to 
$$
(1-p)^f = \exp(-(1+o(1))p f) = \exp \left( - O \left( \frac {d^{i/2}}{n}  \right) \right) = 1 - O \left( \frac {d^{i/2}}{n} \right).
$$

Part (iii) is a straightforward implication of (i). In order to have good bounds on the ratios of the cardinalities of $N(v,1)$, $N(v,2)$, and so on, we consider the Breadth First Search (BFS) algorithm that explores vertices one by one (instead of the whole $j$-th neighbourhood). Formally, the process is initiated by putting $v$ into the queue $Q$. In each step of the algorithm, one vertex $w$ is taken from $Q$ and edges from $w$ to all vertices that are not in $F$ and have not yet been discovered are examined. All new neighbours of $w$ that are found are put into the queue $Q$. The process continues until the queue $Q$ is empty or vertices of $N_{V \setminus F}(v,i/2-1)$ are discovered. (Note that if the process stops because $N_{V \setminus F}(v,i/2-1)$ is discovered, vertices of $S_{V \setminus F}(v,i/2-1)$ are in the queue $Q$; that is, no vertex from this sphere is processed and, in particular, edges in the graph induced by $S_{V \setminus F}(v,i/2-1)$ are not exposed yet.)

Suppose that $N_{V \setminus F}(v,j-1)$ is discovered and we continue investigating vertices of the sphere $S_{V \setminus F}(v,j-1)$, one by one, that are in the queue $Q$. Provided $O(d^{i/2-1})$ vertices have been discovered so far, it follows from part (i) that we may assume that when each vertex of $S_{V \setminus F}(v,j-1)$ is processed, we discover $d (1+o(\log \log n / \log^2 n))$ new neighbours that belong to $S_{V \setminus F}(v,j)$. After that we update $F$ by adding all newly discovered vertices to it, adding the vertex processed at this step, and removing the next vertex to be processed---see Figure~\ref{PictureF2}.
\begin{figure}
  \includegraphics[width=0.4\textwidth]{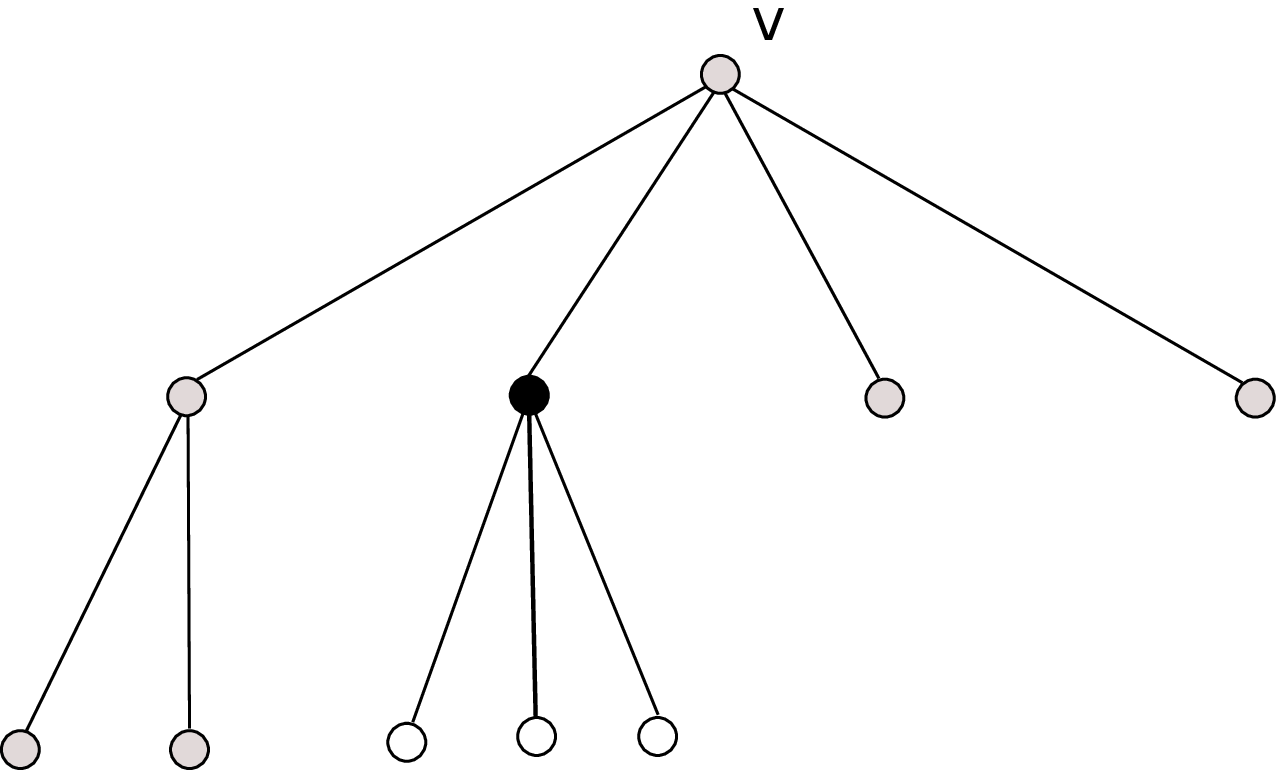}
  \includegraphics[width=0.4\textwidth]{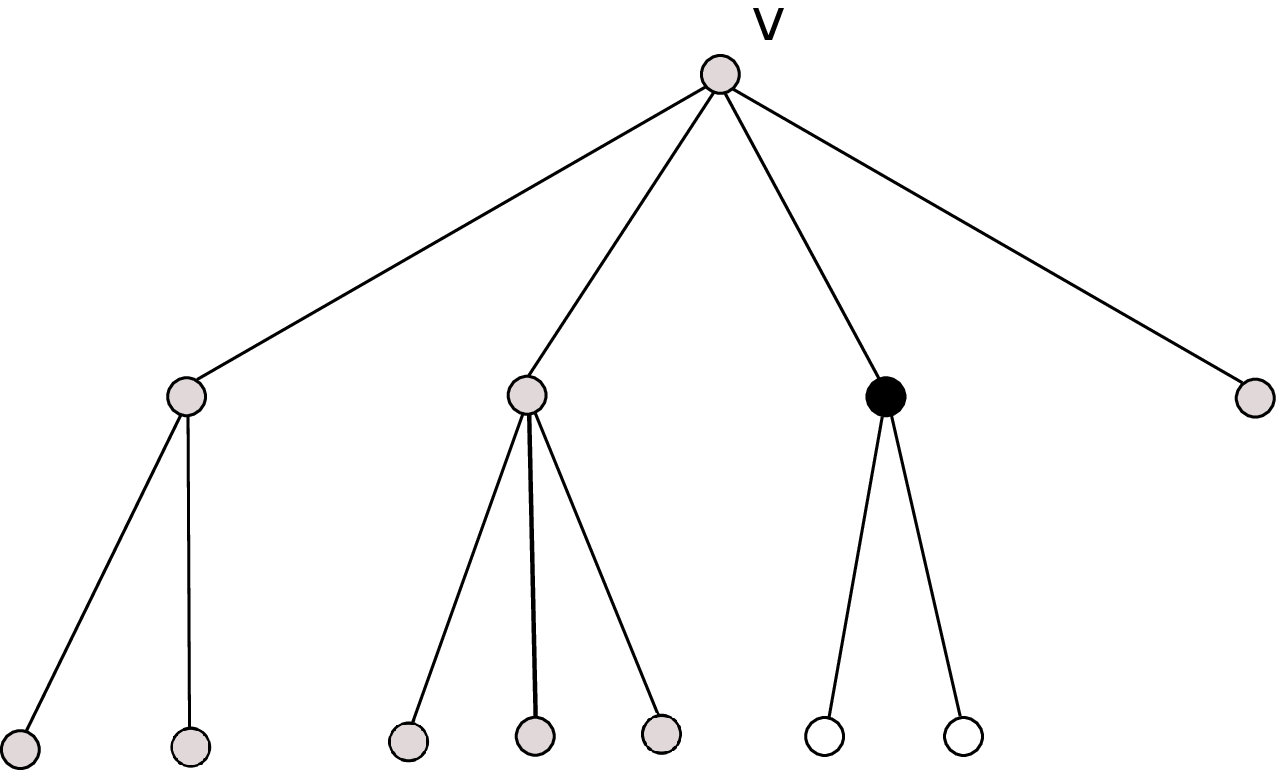}
  \caption{Two consecutive steeps of BFS started from vertex $v$. The black vertex is the vertex currently exposed. Grey vertices form the set $F$ that is updated each time. White vertices are newly discovered ones.}\label{PictureF2}
\end{figure}
We continue until $N_{V \setminus F}(v,j)$ is discovered to get that 
$$
|S_{V \setminus F}(v,j)| = |S_{V \setminus F}(v,j-1)| d (1+o(\log \log n / \log^2 n)). 
$$
We consider this up to the $j$'th iterated neighbourhood, where $j=i/2-1$ and $d^{i-1} \le n \log n / 16$ and thus  $j = O(\log n /\log \log n)$. Then the cumulative multiplicative error term  is 
$$
\left( 1 + o \left( \frac {\log \log n}{\log^2 n} \right) \right)^j = (1 + o(\log^{-1} n)),
$$
and thus $|S(v,j)|=d^j(1 + o(\log^{-1} n))$, and 
$$
|N(v,j)|=\sum_{i=1}^j |S(v,i)|=|S(v,j)|(1+O(1/d)).
$$
This establishes (iii).

For parts (iv), (v), and (vi) we note that in each step of the BFS algorithm the probability that there is no edge from $w$ (the vertex that is processed at this point) to vertices that have been already discovered is, by part (ii), $1-O(d^{i/2}/n)$. 
Hence, by the union bound, a.a.s.\ this never happens since the number of vertices processed is 
\begin{eqnarray*}
|N_{V \setminus F}(v,i/2-2)| &=& (1+o(1)) d^{(i-1)/2-3/2} = O(d^{-3/2} \sqrt{n \log n}) \\
&=& O(d^{-3/2} (n \log n) / d^{(i-1)/2}) = O(d^{-1} (n \log n) / d^{i/2}) \\
&=& o(n/d^{i/2}).
\end{eqnarray*}
The claim follows.

Part (vii) follows immediately (and deterministically) from (iv), (v), and (vi). The proof of the lemma is finished.
\end{proof}

\bigskip

We first give the proof of the result for very dense graphs.

\begin{proof}[Proof of Theorem~\ref{thm:main}(iv)-(v).]
For $p =1-o(1/n^2)$, note that the expected number of edges in the complement of $G$ is ${n \choose 2} (1-p) = o(1)$, and thus by Markov's inequality, a.a.s., $G$ is the complete graph on $n$ vertices. If this is the case, then $d(u,v)=1$ for any pair of vertices $u$ and $v$, and thus for any four vertices $u,v,x,y$, clearly, $d(u,v)+d(x,y)-(d(u,x)+d(v,y))=0$, and hence $\delta_H(G)=0$. Part (v) is proved.	

For $p=1-2c/n^2$ and $c > 0$, note that a.a.s.\ there is no component of size $3$ or more in the complement of $G$. Thus, a.a.s., for all four-tuples of vertices in the original graph, either all edges are present, only one edge is missing, or two disjoint edges are missing. In all of these cases, the non-adjacent vertices are at distance $2$, and thus a.a.s.\ $\delta_H(G) \leq 1$.
The expected number of edges in the complement of $G$ equals ${n \choose 2}(1-p) = (1+o(1)) c$. Also, for any fixed $r$, the $r$-th moment of the number of edges in the complement of $G$ equals $c^r(1+o(1))$, and thus, by the method of moments (see, for example, Theorem~1.22 of~\cite{bol}) the number of edges converges to a random variable with a Poisson distribution with parameter $c$. In particular, with probability $(1+o(1))e^{-c}$, the complement of $G$ is empty, and by the argument in the first case, we have $\delta_H(G)=0$. Also, with probability $(1+o(1))ce^{-c}$, the complement of $G$ contains exactly one edge, say $\{u,v\}$. For the four-tuples not containing both $u$ and $v$, the analysis is as before. For a four-tuple $u,v,x,y$ we now have
for the distances in the original graph $d(u,v)+d(x,y)=3$, $d(u,x)+d(v,y)=d(u,y)+d(v,x)=2$, and thus $\delta_H(G)=1/2$. Finally, with probability $(1+o(1))\left(1-(c+1)e^{-c}\right)$ the complement of $G$ has at least two disjoint edges, say $\{u,v\}$ and $\{x,y\}$. In this case, in the original graph we have $d(u,v)+d(x,y)=4$, $d(u,x)+d(v,y)=d(u,y)+d(v,x)=2$, and thus $\delta_H(G) \geq 1$, and part (iv) follows.
\end{proof}

The main challenge of this paper is to prove the following result and the whole next section is dedicated to it. Here, we show how Theorem~\ref{thm:main}(i)-(iii) can be derived from it.

\begin{theorem}\label{thm:main2}
Suppose that 
$$
d=d(n)=p(n-1) \gg \frac {\log^5 n}{(\log \log n)^2} \textrm{ \ \ \ and \ \ \  } p = 1 - \omega(1/n^2).
$$  
Let $i \geq 2$ be the largest even integer such that $d^{i-1} \leq \frac{1}{16} n \log n.$ 
Let $G \in \mathcal{G}(n,p)$. Then, a.a.s., $\delta_H(G) \geq i/2$. 
\end{theorem}

\begin{proof}[Proof of Theorem~\ref{thm:main}(i)-(iii)]
Fix $j$ to be the smallest integer such that $d^j/n- 2\log n \to \infty$. In particular, $d^{j+1}/n = \omega(\log n)$. Moreover, it follows from Corollary~\ref{cor:diameter} that the diameter of $G$ is at most $j$ a.a.s. Hence, by Lemma~\ref{lem:upperbound}, a.a.s.\ $\delta_H(G) \leq j/2$. This establishes upper bounds in parts (i) and (ii).

Suppose first that $j$ is even and that $d^{j-1} \leq \frac{1}{16} n \log n$. Then $j$ is the largest even integer such that  $d^{j-1} \leq \frac{1}{16} n \log n.$ By Theorem~\ref{thm:main2}, a.a.s.\ $\delta_H(G) \geq j/2$ and part (i) holds.

Suppose next that $j$ is even and that $d^{j-1} > \frac{1}{16} n \log n$ (note that it follows from the definition of $j$ that $d^{j-1} \leq (2+o(1))n \log n$). Then $j-2$ is the largest even integer such that $d^{j-3} \leq \frac{1}{16} n \log n$, and by Theorem~\ref{thm:main2}, a.a.s.\ $\delta_H(G) \geq j/2 -1$. This finishes part (ii).

Finally, suppose that $j$ is odd. Since $d^{j-1}/n = O(\log n)$, $d^{j-2}/n = o(\log n)$, and thus $j-1$ is the largest even integer such that $d^{j-2}\leq \frac{1}{16}n \log n.$ By Theorem~\ref{thm:main2}, a.a.s., $\delta_H(G) \geq (j-1)/2$. Since a.a.s.\ the diameter of $G$ is at most $j$, and $j$ is odd, by Lemma~\ref{lem:upperbound2} we have that a.a.s.\ $\delta_H(G) \leq (j-1)/2$. Part (iii) and so the whole proof is finished.
\end{proof}

\section{Proof of Theorem~\ref{thm:main2}}
Let $G = (V, E) \in \mathcal{G}(n,p)$ and suppose that $d=p(n-1) \gg \frac {\log^5 n}{(\log \log n)^2}$ and $p = 1 - \omega(1/n^2)$.  Let $i \geq 2$ be the largest even integer such that $d^{i-1} \leq \frac{1}{16} n \log n.$ Assume first that $d > (\frac{1}{16}n\log n)^{1/3}$ which implies that $i=2$. In this case, we have to prove that a.a.s.\ $\delta_H(G) \geq 1$. It therefore suffices to find four vertices $u,v,x,y$ such that the subgraph induced by them is a $4$-cycle. Since $p > n^{-2/3}(\frac{1}{16}\log n)^{1/3}$ and $1-p = \omega(n^{-2})$, the expected number of induced cycles of length 4 is ${n \choose 4} p^4 (1-p)^2 \to \infty$. It is a straightforward application of the second moment method to show that a.a.s.\ there is at least one induced cycle in $G$ and the statement follows in this case.

Hence, from now on we may assume that
\begin{equation}\label{eq:deg_bound}
d \leq \left(\frac{1}{16}n\log n\right)^{1/3},
\end{equation}
which implies that $i \geq 4$. We need one more definition: for a given $u \in V$, $k \geq 1$, and $A \subseteq V$, we say that $N_{V \setminus A} (u,k)$ \textbf{expands well} if for all $j=1,2,\ldots, k$, $$|N_{V \setminus A} (u,j)| =|S_{V \setminus A}(u,j)|(1+O(1/d))=d^{j}(1+o(\log^{-1} n))$$ and for all $j=1,2,\ldots, k-1$, every vertex of $S_{V \setminus A}(u,j)$ has $d(1+o(\log \log n / \log^{2} n))$ neighbours in $S_{V \setminus A}(u, j+1)$.  Finally, fix a four-tuple of different vertices $u,v,x,y$ and consider the following process (see Figure~\ref{FigHyp1}): \\
\begin{figure}
  \includegraphics[width=.4\textwidth]{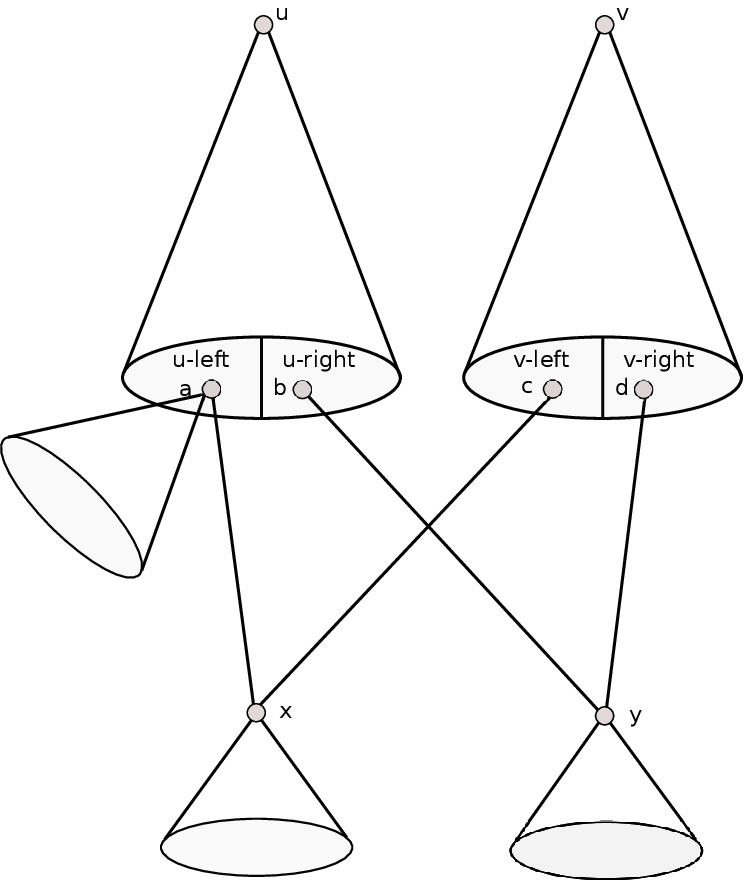} 
  \includegraphics[width=.4\textwidth]{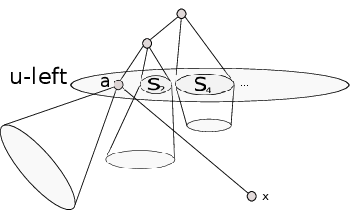} 
\caption{ left: \textsc{Hyperbolicity($u,v,x,y$)}, the big picture; right: the neighbourhood exposure around $a$ in more detail}\label{FigHyp1}
\end{figure}
 \textsc{Hyperbolicity($u,v,x,y$)}
 \begin{enumerate}
 \item Let $B:=\{v,x,y\}$. Perform Breadth-First Search (BFS) from $u$ in the graph induced by $V \setminus B$ to expose $N_{V \setminus B}(u,i/2 - 1)$. Make sure that the following properties hold (otherwise stop the process): 
   \begin{description}
   \item [(a)] $N_{V \setminus B}(u,i/2 - 1)$ expands well. 
   \item [(b)] The graph induced by $N_{V \setminus B}(u,i/2-1)$ is a tree.
   \item [(c)] There is no edge from $N_{V \setminus B}(u,i/2-2)$ to $\{v,x,y\}$.
      \end{description} 
 (As a result, $N(u,i/2-1)=N_{V \setminus B}(u,i/2-1)$ and so $N(u,i/2-1)$ expands well, $N(u,i/2-1)$ is a tree, and $\{v,x,y\} \cap N(u,i/2-1) = \emptyset$.)
 \item Let $D:=N(u,i/2 - 1) \cup \{x,y\}$. Perform BFS from $v$ in the graph induced by $V \setminus D$ to expose $N_{V \setminus D}(v,i/2-1)$. (The reason that here we restrict ourselves to the induced graph is to make sure no edge in this graph is already exposed and so, as typical, we perform BFS by exposing edges one by one, as required.) Make sure that the following properties hold (otherwise stop):
   \begin{description}
      \item [(d)]  $N_{V \setminus D}(v,i/2-1)$ expands well. 
      \item [(e)] There is no edge from $N_{V \setminus D}(v,i/2-1)$ to $S(u,i/2-1)$ (note that edges from vertices of $N(u,i/2-2)$ are already exposed, so that the only chance for the intersection of $N_{V \setminus D}(v,i/2-1)$ and $N(u,i/2-1)$ to be non-empty is when we reach vertices of $S(u,i/2-1)$). 
      \item [(f)] The graph induced by $N_{V \setminus D}(v,i/2-1)$ is a tree. 
      \item [(g)] There is no edge from $N_{V \setminus D}(v, i/2-2)$ to $\{x,y\}$. 
   \end{description}  
 (As a result, $N(v,i/2-1) = N_{V \setminus D}(v,i/2-1)$ and $N(v,i/2-1)$ so expands well, $N(v,i/2-1) \cap N(u,i/2-1) = \emptyset$, $N(v,i/2-1)$ is a tree, and $\{x,y\} \cap N(v,i/2-1) = \emptyset$.)
 \item Let us partition (arbitrarily) the neighbours of $u$ into two sets, $U_L$ and $U_R$, such that $||U_L| - |U_R|| \le 1$. Let us partition the vertices of $S(u,i/2-1)$ and call vertices that are at distance $i/2-2$ from $U_L$ to be `$u$-left'; otherwise, they are called `$u$-right'. Similarly, the vertices of $S(v,i/2-1)$ are partitioned into $v$-left and $v$-right ones.
Expose the edges between $x$ and $S(u,i/2-1) \cup S(v,i/2-1)$ and similarly also between $y$ and $S(u,i/2-1) \cup S(v,i/2-1)$. 
Make sure that the following properties hold (otherwise stop):
 \begin{description}
 \item[(h)]  The number of $u$-left vertices is $\frac {d^{i/2-1}}{2} (1+o(\log^{-1} n))$, and the number of $u$-right vertices is also $\frac {d^{i/2-1}}{2} (1+o(\log^{-1} n))$.
 \item[(i)]  The number of $v$-left vertices is $\frac {d^{i/2-1}}{2} (1+o(\log^{-1} n))$, and the number of $v$-right vertices is also $\frac {d^{i/2-1}}{2} (1+o(\log^{-1} n))$.
 \item[(j)] 
 There is exactly one edge between $x$ and the $u$-left vertices; call the corresponding neighbour of $x$ to be $a$. There is no edge between $x$ and the $u$-right vertices.
 \item[(k)]  
  There is exactly one edge between $x$ and the $v$-left vertices; call the corresponding neighbour of $x$ to be $b$. There is no edge between $x$ and the $v$-right vertices.
  \item[(l)]  There is exactly one edge between $y$ and the $u$-right vertices; call the corresponding neighbour of $y$ to be $c$. There is no edge between $y$ and the $u$-left vertices.
 \item[(m)]  
  There is exactly one edge between $y$ and the $v$-right vertices; call the corresponding neighbour of $y$ to be $d$. There is no edge between $y$ and the $v$-left vertices.
 \end{description}
 \item In this step, the neighbourhood of $a$ is investigated. Unfortunately, this is slightly more complicated since some part of the neighbourhood of $a$ is already ``buried'' in $N(u,i/2-1)$. In order to accomplish our goal, we need to perform BFS not only from $a$ (up to level $i/2-2$), but also from some other vertices of $S(u,i/2-1)$ (this time going not as deep as $i/2-2$; the level until which the neighborhood is explored depends on the distance from $a$)---see Figure~\ref{FigHyp1} (right side).
 
  Formally, for $1 \le k \le i/2-2$, let $S_k$ be the set of vertices of $S(u,i/2-1)$ that are at distance $k$ from $a$ in the tree induced by $N(u,i/2-1)$. (In fact, $k$ has to be even in order for $S_k$ to be non-empty, but we consider all values of $k$ for simplicity.)
Let 
$$
F:= \left( N(u,i/2-1) \setminus \left( \{a\} \cup \bigcup_{k=1}^{i/2-2} S_k \right) \right) \cup N(v,i/2-1) \cup \{x\} \cup \{y\}.
$$ 
We perform BFS from $a$ and from vertices of $\bigcup_{k=1}^{i/2-2} S_k$ in the graph induced by $V \setminus F$; we reach vertices at distance $i/2-2$ from $a$ and at distance $i/2-2-k$ from $S_k$. Make sure that the following properties hold (otherwise stop).
 \begin{description}
 \item[(n)]  $N_{V \setminus F}(a,i/2-2)$ expands well. Moreover, for all $1 \le k \le i/2-2$ and all $\ell \in S_k$ we have that $N(\ell, i/2-2-k)$ expands well. In particular, 
 \begin{eqnarray*}
 |N_{V \setminus F}(a,i/2-2)| &=& d^{i/2-2} (1+o(\log^{-1} n))\\
  \sum_{k=1}^{i/2-2} \sum_{\ell \in S_k} |N_{V \setminus F}(\ell,i/2-2-k)| &=& o(d^{i/2-2} \log^{-1} n).
\end{eqnarray*}

\item[(o)] There is no edge from $N_{V \setminus F}(a,i/2-2) \setminus \{a\}$ to $F$ and for every  $k=1, 2, \ldots,i/2-2$ and every vertex $\ell \in S_k$, there is no edge from $N_{V \setminus F}(\ell, i/2-2-k) \setminus \{\ell\}$ to $F$. 
\item[(p)]    All graphs exposed in this step are disjoint trees. Note that this implies that $N(a,i/2-2)$ is a tree. 
\item[(q)]  
For $1 \leq k \leq i/2-2$, let  $S'_k$ be the set of vertices of $S(v,i/2-1)$ that are at distance $k$ from $b$ in the tree induced by $N(v,i/2-1)$ and let
\begin{eqnarray*}
F' &:=& \left( N(v,i/2-1) \setminus \left( \{b\} \cup \bigcup_{k=1}^{i/2-2} S'_k \right) \right) \cup N(u,i/2-1) \cup \{x\} \cup \{y\} \\
&& \cup N(a,i/2-2).
\end{eqnarray*}
Perform BFS from $b$ and from vertices of $\bigcup_{k=1}^{i/2-2} S'_k$ in the graph induced by $V \setminus F'$; Properties \textbf{(n)}, \textbf{(o)},  and \textbf{(p)} hold when $a$ is replaced by $b$, $F$ is replaced by  $F'$, and the sets $S_k$ are replaced by $S_k'$. 
\item[(r)] 
For $1 \leq k \leq i/2-2$, let  $S''_k$ be the set of vertices of $S(u,i/2-1)$ that are at distance $k$ from $c$ in the tree induced by $N(u,i/2-1)$ and let
\begin{eqnarray*}
F'' &:=& \left( N(u,i/2-1) \setminus \left( \{c\} \cup \bigcup_{k=1}^{i/2-2} S''_k \right) \right) \cup N(v,i/2-1) \cup \{x\} \cup \{y\} \\
&& \cup N(a,i/2-2) \cup N(b,i/2-2).
\end{eqnarray*} 
Perform BFS from $c$ and from vertices of $\bigcup_{k=1}^{i/2-2} S''_k$ in the graph induced by $V \setminus F''$; Properties \textbf{(n)}, \textbf{(o)},  and \textbf{(p)} hold when $a$ is replaced by $c$, $F$ is replaced by  $F''$, and the sets $S_k$ are replaced by $S''_k$. 
\item[(s)]  
For $1 \leq k \leq i/2-2$, let  $S'''_k$ be the set of vertices of $S(v,i/2-1)$ that are at distance $k$ from $d$ in the tree induced by $N(v,i/2-1)$ and let
\begin{eqnarray*}
F''' &:=& \left( N(v,i/2-1) \setminus \left( \{d\} \cup \bigcup_{k=1}^{i/2-2} S'''_k \right) \right) \cup N(u,i/2-1) \cup \{x\} \cup \{y\} \\
&& \cup N(a,i/2-2) \cup N(b,i/2-2) \cup N(c,i/2-2).
\end{eqnarray*} 
Perform BFS from $d$ and from vertices of $\bigcup_{k=1}^{i/2-2} S'''_k$ in the graph induced by $V \setminus F'''$; Properties \textbf{(n)}, \textbf{(o)},  and \textbf{(p)} hold when $a$ is replaced by $d$, $F$ is replaced by  $F'''$, and the sets $S_k$ are replaced by $S'''_k$. 
 \end{description}
 \item Let 
\begin{eqnarray*}
Q &:=& N(u,i/2-1) \cup N(v,i/2-1) \cup \{y\} \\
&& \cup N(a,i/2-2) \cup N(b,i/2-2) \cup N(c,i/2-2) \cup N(d,i/2-2). 
\end{eqnarray*}
We perform BFS from $x$ in the graph induced by $V \setminus Q$ to expose $N_{V \setminus Q}(x,i/2-1)$. Make sure that the following properties hold (otherwise stop):
 \begin{description}
    \item[(t)] $N_{V \setminus Q}(x,i/2-1)$ expands well. 
        \item[(u)]  There is no edge from $N_{V \setminus Q}(x,i/2-1) \setminus \{x\}$ to $Q$. 
    \item[(v)]  $N_{V \setminus Q}(x,i/2-1)$ is a tree. Note that since  the remaining branches accessed by the edge $xa$ and the edge $xb$ are already guaranteed to be trees and there are no edges between different parts, this implies that $N(x,i/2-1)$ is a tree.
\item[(w)] There is no edge between $x$ and $y$. 
\item[(x)] Let $R:=Q \cup N(x,i/2-1) \setminus \{y\}$. Properties \textbf{(t)}, \textbf{(u)},  and \textbf{(v)} hold when $x$ is replaced by $y$ and $Q$ is replaced by  $R$. 
\end{description}
\item It is the end of this tedious process so it is time for a short break---perform a fireworks show (fireworks are explosive pyrotechnic devices typically used for aesthetic, cultural, and religious purposes; here the main purpose is to celebrate finding an object with the desired properties). 
\end{enumerate}

\bigskip
We say that the process \textsc{Hyperbolicity}(u,v,x,y) \textbf{terminates successfully} if all the required conditions are satisfied, that is, the process does not stop prematurely before reaching the end.

\begin{claim}\label{claim:correctness}
If \textsc{Hyperbolicity}(u,v,x,y) terminates successfully for some four-tuple $u,v,x,y$, then $\delta_H(G) \geq i/2$.
\end{claim}
\begin{proof}
By Property \textbf{(e)}, $N(u,i/2-1) \cup N(v,i/2-1) = \emptyset$, and there is no edge between $S(u,i/2-1)$ and $S(v,i/2-1)$. Hence, we have $d(u,v) \geq i$. In fact, the distance between $u$ and $v$ is exactly $i$, since by Properties $\textbf{(j)}$ and $\textbf{(k)}$ there is a path of length $i$ going through $x$. Next, by Properties $\textbf{(c)}$, $\textbf{(g)}$, $\textbf{(j)}$, $\textbf{(k)}$, $\textbf{(l)}$ and $\textbf{(m)}$, we have $d(u,x)=d(v,x)=d(u,y)=d(v,y)=i/2$. For the distance between $x$ and $y$ observe the following: first, by Properties $\textbf{(j)}$ and $\textbf{(l)}$ there is an $x-y$ path of length $i$ going from $x$ to $a$, then through $u$ to $c$, and then to $y$. We will show that $N(x,i/2-1) \cap N(y,i/2-1) = \emptyset$ and that there is no edge between $S(x,i/2-1)$ and $S(y,i/2-1)$.  Indeed, if a shortest $x-y$-path first goes from $x$ to $a$, by Properties $\textbf{(o)}$, $\textbf{(r)}$, $\textbf{(s)}$ and $\textbf{(x)}$, it has to go until $S(a,i/2-2)$, and then it has to pass through at least two more edges before entering $S(y,i/2-1)$, including $S(c,i/2-2)$ and $S(d,i/2-2)$, and in each case the length is at least $i$. By properties $\textbf{(q)}$, $\textbf{(r)}$, $\textbf{(s)}$ and $\textbf{(x)}$, the same holds if the path starts from $x$ to $b$. If the path from $x$ neither goes through $a$ nor through $b$, it has to go through $N_{V \setminus Q}(x,i/2-1)$.  By Properties $\textbf{(u)}$ and $\textbf{(x)}$, it has to arrive at $S(x,i/2-1)$, and then it has to go through at least two edges before entering  $S(y,i/2-1)$, including $S(c,i/2-2)$ and $S(d,i/2-2)$, and in each case the length is also at least $i$.

Hence, $d(u,v)+d(x,y)\geq 2i$, $d(u,x)+d(v,y)=d(u,y)+d(v,x)=i$, and thus, $\delta_H(G) \geq i/2$. The proof of the claim is finished.
 \end{proof}

Thus, by Claim~\ref{claim:correctness}, in order to show that a.a.s.\ $\delta_H(G) \geq i/2$, it suffices to show that a.a.s.  \textsc{Hyperbolicity}($u,v,x,y$) succeeds for at least one four-tuple of vertices $u,v,x,y$. Let $X_{u,v,x,y}$ be the indicator random variable defined as follows:
$$
X_{u,v,x,y}=\begin{cases} 1, & \mbox{ if \textsc{Hyperbolicity}}(u,v,x,y) \mbox{ terminates successfully,}\\ 
0,  & \mbox{ otherwise. } \end{cases}
$$
Let 
$$
 X=\sum_{u,v,x,y} X_{u,v,x,y}, 
$$
where the sum is taken over all $\binom{n}{4}$ $4$-tuples of all disjoint vertices. 
In order to prove that a.a.s. $\delta_H(G) \geq i/2$, we will apply the second moment method to $X$.
Define $q=\exp(-d^{i-1}/(2n))$ and note that  from the assumption that $d^{i-1} \leq \frac{1}{16}n \log n$ we have $q \geq n^{-1/32}$.

\begin{lemma}\label{lem:firstmoment}
For a fixed four-tuple of vertices $u,v,x,y$, 
$$\Prob{(X_{u,v,x,y}=1)} = \left( \frac {d^{i/2}}{2n} \right)^4q^{16} (1+o(1)).
$$
Moreover,  
$$\E{X}  = \frac{d^{2i}}{384}q^{16}(1+o(1)) = \Omega(n^{5/6}(\log n)^{4/3}).$$
\end{lemma}
\begin{proof}
Fix a four-tuple of vertices $u,v,x,y$. First, we will calculate $\Pr{(X_{u,v,x,y}=1)}$. We will estimate for each of the five steps of \textsc{Hyperbolicity}(u,v,x,y) the probability that it fails at that step. For $z \in \{a,b,\ldots,x\}$, let $P_z$ be the indicator random variable for the event that Property \textbf{(z)} succeeds provided that all previous properties have succeeded as well. Similarly, for step $z \in \{1,2,\ldots,5\}$, let $T_z$ be the indicator random variable for the event that step $z$ succeeds provided that all previous steps have succeeded as well.

By Lemma~\ref{lem:expansion}(iii) and (iv),
$$
\Prob{(P_a=1)}=1+o(1).
$$ 
Let $E$ be the event that there is no edge within  the last sphere $S_{V \setminus B}(u,i/2-1)$. By Lemma~\ref{lem:expansion}(v) and (vi), in order to calculate the probability that Property \textbf{(b)} holds, it remains to estimate the probability that $E$ holds. We have
\begin{eqnarray}
\Prob{(E)}&=&(1-p)^{\binom{|S_{V \setminus B}(u,i/2-1)|}{2}} \label{propE} \\
&=& \exp \left(-p(1+O(p)) \binom{|S_{V \setminus B}(u,i/2-1)|}{2} \right)  \nonumber \\
&=&\exp\left(-p(1+O(p)) \left( \frac {(d^{i/2-1})^2}{2} (1+o(\log^{-1}n)) \right) \right) \nonumber,
\end{eqnarray}
where the last equality follows from Property \textbf{(a)} that is assumed to hold deterministically now.
Hence,
 \begin{eqnarray*}\label{eq:EventE}
\Prob{(E)}&=& \exp\left(-\frac{pd^{i-2}}{2} (1+o(\log^{-1} n)) \right) \\
&=& \exp\left( - \frac{d^{i-1}}{2n}(1+o(\log^{-1} n))\right) \\ 
&=& \exp\left (-\frac{d^{i-1}}{2n}\right) (1+o(1)) = q(1+o(1)),  
\end{eqnarray*}
where the last line follows from the assumption that $d^{i-1} \leq \frac{1}{16}n \log n$. 
 Hence, the probability that $N_{V \setminus B}(u,i/2-1)$ is a tree is asymptotically equal to the probability that the event $E$ holds, and thus 
$$
\Prob{(P_b=1)}=q(1+o(1)).  \\
$$
Now, let us move to Property \textbf{(c)}. By Lemma~\ref{lem:expansion}(ii) together with a union bound over all vertices in $N_{V \setminus B}(u,i/2-2)$, we see that with probability 
$$
1-O(d^{i/2-2}d^{i/2}/n)=1-O(d^{i-2}/n)=1+o(1)
$$ 
there is no edge from $N_{V \setminus B}(u,i/2-2)$ to $B$, and thus $$\Prob{(P_c=1)}=1+o(1).$$ Hence,
$$
\Prob{(T_1=1)}=q(1+o(1)).
$$

Next, for Property \textbf{(d)}, by Lemma~\ref{lem:expansion}(iii) and (iv), we obtain 
$$
\Prob{(P_d=1)}=1+o(1).
$$  
For Property \textbf{(e)}, since Property \textbf{(d)} is assumed to hold deterministically at this point, we have 
\begin{eqnarray*}
\Prob{(P_e=1)}&=&(1-p)^{|S(u,i/2-1)||N_{V \setminus D}(v,i/2-1)|} \\
&=&(1-p)^{|S(u,i/2-1)||S_{V \setminus D}(v,i/2-1)|(1+O(1/d))} \\
&=&(1-p)^{(d^{i/2-1})^2 (1+o(\log^{-1} n))} \\
&=&\left(\exp\left( -p(1+O(p))\left( \frac{(d^{i/2-1})^2}{2}(1+o(\log^{-1} n)) \right) \right)\right)^2 
\end{eqnarray*}
and hence, by the same calculations following~\eqref{propE} we obtain 
$$
\Prob{(P_e=1)}=q^2(1+o(1)).
$$
The probability of having Property \textbf{(f)} is calculated as before for Property \textbf{(b)}, and of having Property \textbf{(g)} as before for Property \textbf{(c)}.
Thus,
$$
\Prob{(T_2=1)} = q^3 (1+o(1)).
$$

For Property \textbf{(h)} we immediately have by Lemma~\ref{lem:expansion}(vii) that 
$$
\Prob{(P_h=1)} = 1+o(1),
$$
and the same applies to Property \textbf{(i)}.
For Property \textbf{(j)}, since Property \textbf{(h)} is assumed to hold deterministically, we have 
\begin{eqnarray*}
\Prob{(P_j=1)} &=& (1+o(1)) \frac {d^{i/2-1}}{2} p(1-p)^{(1+o(1))d^{i/2-1}} \\
&=&  (1+o(1))\frac{d^{i/2}}{2n}\exp\left(-p(1+O(p))(1+o(1))d^{i/2-1}\right) \\
&=&  (1+o(1))\frac{d^{i/2}}{2n}\exp\left(-(1+o(1))d^{i/2}/n\right)  \\
&=&  (1+o(1))\frac{d^{i/2}}{2n}\exp\left(-(1+o(1))(d^{i-1})^{1/2}\sqrt{d}/n\right)  \\
&=&  (1+o(1))\frac{d^{i/2}}{2n},
\end{eqnarray*}
where the last line follows from the fact that $(d^{i-1})^{1/2} = O(\sqrt{n \log n})$ (by definition of $i$), and by~\eqref{eq:deg_bound}, which implies that $(d^{i-1})^{1/2}\sqrt{d}/n =O(\sqrt{d \log n/n})=o(1)$. 
Note then that Properties \textbf{(j)}, \textbf{(k)}, \textbf{(l)} and \textbf{(m)}  are symmetric and mutually independent, and thus they are calculated in the same way. Therefore
$$
\Prob{(T_3=1)} = \left(\frac{d^{i/2}}{2n}\right)^4(1+o(1)).
$$

Let us move to investigating Properties \textbf{(n)}, \textbf{(o)}, and \textbf{(p)}. First, we perform BFS from $a$ in $V \setminus F$. It follows immediately from Lemma~\ref{lem:expansion}(iii) that  $N_{V \setminus F}(a,i/2-2)$ expands well and so the bound on $|N_{V \setminus F}(a,i/2-2)|$ in Property \textbf{(n)} holds a.a.s. For the vertices in $S_k$, since Properties \textbf{(a)} and \textbf{(b)} are assumed to hold deterministically, for every even value of $k$ such that $2 \le k \le i/2-2$, the number of vertices in $S_k$ is $(1+o(1))d^{k/2}$. In order to deal with the second bound of Property \textbf{(n)} (and to investigate Properties \textbf{(o)} and \textbf{(p)} at the same time), we mimic the proof of Lemma~\ref{lem:expansion}(iii). We perform BFS from some other vertex in some $S_k$ in $V \setminus F$, updating the set $F$ every time a vertex is processed. As shown in Figure~\ref{PictureF2}, the vertex that was processed before, together with all its neighbours, will be added to $F$, and the next vertex in the queue to be processed will be taken out of $F$. Once we are done, we take the next vertex in some $S_k$ and continue in this way until all neighbourhoods under consideration are discovered. Arguing as in the proof of Lemma~\ref{lem:expansion}(iii), by Lemma~\ref{lem:expansion}(i) together with a union bound over all vertices processed, we obtain the desired bounds for the sizes of neighbourhoods. 
Moreover, by Lemma~\ref{lem:expansion}(ii) together with a union bound over all vertices that are discovered during this step (at most $O(d^{i/2-2})$ vertices), we get that a.a.s.\ at the time when a given vertex was processed there was no edge to already discovered vertices (neither within the same tree where we started BFS from, nor to other trees, nor to the initial set $F$). This deals with Properties \textbf{(o)} and \textbf{(p)}. Finally, it follows that a.a.s.
$$
\sum_{k=1}^{i/2-2} \sum_{\ell \in S_k} |N_{V \setminus F}(\ell,i/2-2-k)| \leq \sum_{k=1}^{i/2-2} (1+o(1))d^{i/2-2-k/2} = o(d^{i/2-2} \log^{-1} n),
$$
where the last equality follows from the fact that $d \gg \log^5 n/(\log \log n)^2$. Thus,  
$$
\Prob{(P_{n}=1 \textrm{ and } P_{o}=1 \textrm{ and } P_{p}=1)}=1+o(1).
$$ 
The probabilities for Properties \textbf{(q)}, \textbf{(r)} and \textbf{(s)} to hold are calculated in exactly the same way as for Properties \textbf{(n)}, \textbf{(o)} and \textbf{(p)}, and hence $$\Prob{(T_4=1)}=1+o(1).$$

Finally, for $T_5$, when exposing $x$, Property \textbf{(t)} is investigated as before. Also, by analogous calculations as for $T_1$, the probability of having no edge to $Q$ (Property \textbf{(u)})  and the one of being a tree (Property \textbf{(v)}), altogether yield  $q^5(1+o(1))$; note that the exponent of $5$ comes from the fact that $N_{V \setminus Q}(x,i/2-1)$ is a tree (giving one $q$), that there is no edge to $S(u,i/2-1)$ (giving $2$ additional factors of $q$), and no edge to $S(v,i/2-1)$ (giving another $2$ additional factors of $q$). Property \textbf{(w)} clearly also holds with probability $1+o(1)$. Similarly, when exposing $y$,  we also have to consider the condition of having no edge to $N_{V \setminus R}(x,i/2-1)$ (Property \textbf{(x)}), giving us another factor of $q^2$, and thus yielding a probability of $q^7(1+o(1))$. Thus,
\begin{eqnarray*}
\Prob{(T_5=1)} &=& q^{12} (1+o(1)).
\end{eqnarray*}

Combining the events $T_1, T_2, \ldots,T_5$, we obtain
\begin{eqnarray*}
\Prob{(X_{u,v,x,y}=1)} &=& (d^{i/2}/(2n))^4q^{16} (1+o(1)),
\end{eqnarray*}
yielding the first part of the lemma. Observing that $i$ is such that  $d^{i+1} > \frac{1}{16} n \log n$,  and therefore $d^i > \frac{1}{16}n \log n/d$, and also using~\eqref{eq:deg_bound}, we obtain 
\begin{eqnarray*}
\E{X}&= & \binom{n}{4}\left(\frac{d^{i/2}}{2n}\right)^4q^{16} (1+o(1)) \\
&=& \frac{d^{2i}}{384}q^{16}(1+o(1)) \\
&>& \frac{(\frac{1}{16}n\log n)^2}{384d^2}q^{16}(1+o(1)) \\
&=& \Omega(n^{3/2}(\log n/d)^2) \\
&=& \Omega(n^{5/6}(\log n)^{4/3}),
\end{eqnarray*}
which finishes the proof of the lemma.
\end{proof}

\bigskip

We now move to the second moment method.
\begin{lemma}\label{lem:secondmoment}
  $\E{X^2}=(1+o(1))(\E{X})^2$. 
\end{lemma}
\begin{proof}
In order to analyze the expected value of  $X^2$  we will consider $8$ different cases. Note that 
$$
\E{X^2} =  \sum_{u,v,x,y} \sum_{u',v',x',y'} \Prob{(X_{u,v,x,y}=1 \wedge X_{u',v',x',y'}=1)},
$$
where both sums range over all $4$-tuples of different vertices. For a fixed $4$-tuple of vertices $u,v,x,y$, it follows from the first part of Lemma~\ref{lem:firstmoment} that 
$$
\Prob{(X_{u,v,x,y}=1)} = (d^{i/2}/(2n))^4q^{16} (1+o(1)).
$$
Conditioning on $X_{u,v,x,y}=1$, our goal is to investigate $\Prob{(X_{u',v',x',y'}=1 |  X_{u,v,x,y}=1)}$.  Note that the lower bound for $\E{X^2}$ trivially holds and so we aim for an upper bound. Therefore, we focus on properties that hold with probability $o(1)$ in the unconditional case (such as Property \textbf{(b)} that holds with probability asymptotic to $q$). We ignore properties that hold with probability $1+o(1)$ in the unconditional case (such as Property \textbf{(a)}) although it might happen that the probability that they hold in the conditional space is smaller (which clearly helps).

Assume that $X_{u,v,x,y}=1$, and let $U$ be the  set of vertices exposed to certify this, that is,
 \begin{eqnarray*}
 U &:=&N(u,i/2-1) \cup N(v,i/2-1) \cup N(x,i/2-1) \cup N(y,i/2-1)  \\
 && \cup N(a,i/2-2) \cup N(b,i/2-2) \cup N(c,i/2-2) \cup N(d,i/2-2).
 \end{eqnarray*}
Since $X_{u,v,x,y}=1$, $|U|=O(d^{i/2-1})$. 
For $z \in \{u,v,x,y,a,b,c,d\}$, we always denote by $z'$ the vertex in \textsc{Hyperbolicity}$(u',v',x',y')$ corresponding to vertex $z$ in \textsc{Hyperbolicity}$(u,v,x,y)$.  

\smallskip

Before we move to investigating cases, let us note that we may assume the following useful properties that will hold for all $4$-tuples of vertices $u,v,x,y$ with corresponding set $U$ a.a.s. 

\smallskip \textbf{Claim 1}: For a vertex $z' \notin U$, the number of edges between $N(z',i/2-1) \setminus U$ and $U$ is $O(\log n)$.
\begin{proof}
Indeed, we may assume that $|N(z',i/2-1) \setminus U| = O(d^{i/2-1})$ and so the expected number of edges between $N(z',i/2-1) \setminus U$ and $U$ is $O(pd^{i-2})=O(d^{i-1}/n)=O(\log n)$. It follows from Chernoff's bound that there exists some constant $C > 0$ such that the probability to have at least $C \log n$ edges of this type is $o(n^{-4})$.
The contribution of all such four-tuples $u',v',x',y'$ to $X^2$ is $o(1)$ and so can be safely ignored.
\end{proof}

\smallskip \textbf{Claim 2}: For a vertex $z' \notin U$, $|N(z',i/2-1) \cap U| = O(d^{i/4-1} \log n)$.
\begin{proof}
Indeed, suppose that there is an edge from $N(z',i/2-1) \setminus U$ to $U$. In fact, since $U$ is exposed during the BFS process, this edge has to be adjacent to a leaf $\ell$ of the graph induced by $U$. We need to estimate the size of $N(\ell, i/2-2) \cap U$, since vertices in $N(z',i/2-1) \cap U$ due to the existence of this edge form a subset of $N(\ell, i/2-2) \cap U$. From the fact that $U$ induces an almost regular tree, it follows that $|N(\ell, i/2-2) \cap U| = O(d^{i/4-1})$. The claim follows from the same Chernoff bound as in Claim~1. 
\end{proof}

Now, we are ready for a case analysis. 

\bigskip \noindent \textbf{Case 1:} $\{u',v',x',y'\} \cap U = \emptyset.$ 

It follows from Claim 2 above that the part of the graph that needs to be exposed in order to check whether  $X_{u',v',x',y'}=1$ intersects with $U$ only in a negligible way. It is straightforward to show that Properties \textbf{(b)}, \textbf{(e)}, \textbf{(f)}, \textbf{(u)}, \textbf{(v)}  and \textbf{(x)}, as well as Properties \textbf{(j)}, \textbf{(k)}, \textbf{(l)} and  \textbf{(m)} 
are up to a $1+o(1)$ factor independent of conditioning on $X_{u,v,x,y}=1$, and their calculations are as in Lemma~\ref{lem:firstmoment}.  We obtain
$$
 \Prob{(X_{u',v',x',y'}=1 |  X_{u,v,x,y}=1)} \leq (d^{i/2}/(2n))^4q^{16} (1+o(1))=\Prob{(X_{u,v,x,y}=1)}(1+o(1)).
$$ 
Clearly, the number of choices of four-tuples $u,v,x,y$ and $u',v',x',y'$ is at most the square of the number of choices for $u,v,x,y$, so the contribution of this case is at most $(1+o(1))\E{X^2}$.
(We will show that the contribution of all other cases is $o(\E{X^2})$ and so the contribution of this case is indeed $(1+o(1))\E{X^2}$.)

\bigskip \noindent \textbf{Case 2:} $\{u',v'\} \cap U \neq \emptyset$ and  $\{x',y'\} \cap U = \emptyset.$ 

Since $\{x',y'\} \cap U = \emptyset$, edges emanating from $x'$ and $y'$ are not exposed yet. Hence, the probabilities that Properties \textbf{(j)}, \textbf{(k)}, \textbf{(l)} and \textbf{(m)} hold are as in the unconditional case. Ignoring all other properties we get
$$
\Prob{(X_{u',v',x',y'}=1 |  X_{u,v,x,y}=1)} \le (d^{i/2}/(2n))^4 (1+o(1)) = O \left( \Prob{(X_{u,v,x,y}=1)} q^{-16} \right).
$$
On the other hand, since at least one of $u',v'$ is in $U$, only a
$$
O(|U|/n) = O(d^{i/2-1}/n) = O \left( \sqrt{\log n / (dn)} \right) = O(n^{-1/2} (\log^{-2} n) (\log \log n) ) = o(q^{16})
$$ 
fraction of four-tuples is considered here. Hence, the contribution of this case to $X^2$ is negligible. 

\bigskip 

From now on we may assume that at least one of $x',y'$ has to be in $U$.

\bigskip \noindent \textbf{Case 3:} $\{u', v'\} \cap U = \emptyset$, and $|\{x',y'\} \cap U|=1$.

By symmetry, we may assume that $x' \in U$ and $y' \notin U$. Since $y' \notin U$, arguing as in the previous case, we get that Properties \textbf{(l)} and \textbf{(m)} hold with probability up to a $1+o(1)$ factor as in the unconditional case. However, it might happen that, say, $N(u',i/2-1) \cap U \neq \emptyset$ and Property \textbf{(j)} holds ``for free''. Intuitively, for this to happen, an edge joining $N_{V \setminus U}(u',i/2-1)$ and $U$ must occur at the right place, which happens with small probability. Moreover, only a small fraction of four-tuples satisfies $x' \in U$. 

More precisely, for Property \textbf{(j)} to hold ``for free'', we must have that $d(u',x')=i/2$, and so there must be an edge between $S_U(x',k)$ and $S_{V \setminus U}(u',i/2-1-k)$ for some $k = 0, 1, \ldots, i/2-1$. By the union bound, this happens with probability 
$$
O(i d^k d^{i/2-1-k} p) = O \left( \frac {d^{i/2}}{n} \cdot \frac {\log n}{\log \log n} \right),
$$
since $i = O(\log n /\log \log n)$. (In fact, with a slightly more delicate argument one can remove the $\log n / \log \log n$ factor but this is not needed.) The same argument applies to Property \textbf{(k)}. (Recall, that we may assume that $N(u',i/2-1)$ and $N(v',i/2-1)$ are disjoint.) Hence, by considering only these four properties we get that
\begin{eqnarray*}
  \Prob{(X_{u',v',x',y'}=1 |  X_{u,v,x,y}=1)} &=& O \left( \left( \frac {d^{i/2}}{n} \right)^4 \left( \frac {\log n}{\log \log n} \right)^2 \right)\\
&=& O \left( \Prob{(X_{u,v,x,y}=1)} q^{-16} \left( \frac {\log n}{\log \log n} \right)^2 \right).
\end{eqnarray*}
On the other hand, since $x' \in U$, only a 
$$
O(|U|/n) = O(n^{-1/2} (\log^{-2} n) (\log \log n) ) = o(q^{16} (\log n / \log \log n)^{-2})
$$ 
fraction of four-tuples is considered here.  Hence, the contribution of this case to $X^2$ is negligible. 

\bigskip \noindent \textbf{Case 4:} $\{u', v'\} \cap U = \emptyset$, and $|\{x',y'\} \cap U|=2$.

Note that if $X_{u',v',x',y'}=1$ then, in particular, $d(x',y')=i$ and so the distance between $x'$ and $y'$ in the graph induced by $U$ is at least $i$. It follows that, for example, an edge between $s \in S_U(x',k)$ and $t \in S_{V \setminus U}(u',i/2-1-k)$ (for some $k = 0, 1, \ldots, i/2-1$) cannot make both Property \textbf{(j)} and Property \textbf{(l)} to hold, since that would imply that $d(x',y') \le d(x',s) + d(s,y') \le i-2$.
Hence the calculations dealing with $y'$ (related to Properties \textbf{(l)} and \textbf{(m)}) are independent of the calculations dealing with $x'$ (related to Properties \textbf{(j)} and \textbf{(k)}). Arguing as in the previous case, we get that
$$
  \Prob{(X_{u',v',x',y'}=1 |  X_{u,v,x,y}=1)}= O \left( \Prob{(X_{u,v,x,y}=1)} q^{-16}(\log n / \log \log n)^4 \right).
$$
Finally, note that only a 
$$
O( (|U|/n)^2 ) = o(1/n) = o(q^{16} (\log \log n / \log n)^4)
$$ 
fraction of four-tuples is considered here, making the contribution of this case negligible.

\bigskip

From now on we may assume that at least one of $u',v'$ has to be in $U$ (and still at least one of $x',y'$).

\bigskip \noindent \textbf{Case 5:} $|\{u', v'\} \cap U| = 1$, and $|\{x',y'\} \cap U|=1$. 

By symmetry, suppose that $u' \in U$, $v' \notin U$, $x' \in U$, $y' \notin U$. Since $y' \notin U$, as before, we get that Properties \textbf{(l)} and \textbf{(m)} hold with probability up to a $1+o(1)$ factor as in the unconditional case.  Now, since $v' \notin U$, by the same calculations as in Case 3, the probability that Property \textbf{(k)} holds
is at most $O \left( (d^{i/2}/n) \log n / \log \log n \right)$. Property \textbf{(j)}, however, might hold deterministically now, so we lose an additional factor of $O(d^{i/2}/n)$. By considering only Properties \textbf{(k)}, \textbf{(l)} and \textbf{(m)}, we get that
\begin{eqnarray*}
  \Prob{(X_{u',v',x',y'}=1 |  X_{u,v,x,y}=1)} &=& O \left( \left ( \frac{d^{i/2}}{n} \right)^3 \frac {\log n}{\log \log n} \right) \\
  &=& O \left( \Prob{(X_{u,v,x,y}=1)} q^{-16} \frac{n}{d^{i/2}} \frac {\log n} {\log \log n}  \right).
\end{eqnarray*}
On the other hand, since $u',x' \in U$, only a 
$$
O((|U|/n)^2) = O( (d^{i/2-1}/n) n^{-1/2} (\log^{-2} n) (\log \log n) ) = o( (d^{i/2}/n) q^{16} (\log \log n / \log n) )
$$ 
fraction of four-tuples is considered here, and so the contribution of this case to $X^2$ is negligible. 

\bigskip \noindent \textbf{Case 6:} $|\{u', v'\} \cap U| = 1$, and $|\{x',y'\} \cap U|=2$. 

By symmetry, suppose that $u' \in U$ and $v' \notin U$. 
By the same argument as in Case 4, the calculations involving Properties related to $x'$ are independent of those related to $y'$.
Since $v' \notin U$, by the same calculations as in Case 3, the probability that Properties \textbf{(k)}  and \textbf{(m)} both hold is at most $O \left( (d^{i/2}/n)^2 (\log n/\log \log n)^2 \right)$.
This time, Properties \textbf{(j)} and \textbf{(l)} might hold deterministically, so we lose an additional factor of $O((d^{i/2}/n)^2).$ 
By considering only Properties \textbf{(k)} and \textbf{(m)} we get that
$$
  \Prob{(X_{u',v',x',y'}=1 |  X_{u,v,x,y}=1)} = O \left( \Prob{(X_{u,v,x,y}=1)} q^{-16} \left( \frac{n}{d^{i/2}} \right)^2 \left( \frac {\log n} {\log \log n} \right)^2 \right).
$$
On the other hand, since $u',x',y' \in U$, only a 
$$
O((|U|/n)^3) = O( (d^{i/2-1}/n)^2 n^{-1/2} (\log^{-2} n) (\log \log n) ) = o( (d^{i/2}/n)^2 q^{16} (\log \log n / \log n)^2 )
$$ 
fraction of four-tuples is considered here, and the contribution of this case to $X^2$ is negligible. 

\bigskip \noindent \textbf{Case 7:} $|\{u', v'\} \cap U| = 2$, and $|\{x',y'\} \cap U|=1$. 

By symmetry, suppose that $x' \in U$ and $y' \notin U$. Since $y' \notin U$, as before, we get that Properties \textbf{(l)} and \textbf{(m)} hold with probability up to a $1+o(1)$ factor as in the unconditional case. By considering only these two properties we get that
$$
  \Prob{(X_{u',v',x',y'}=1 |  X_{u,v,x,y}=1)} = O \left( \Prob{(X_{u,v,x,y}=1)} q^{-16} \left( \frac{n}{d^{i/2}} \right)^2 \right).
$$
On the other hand, since $u',v',x' \in U$, only a 
$$
O((|U|/n)^3) = O( (d^{i/2-1}/n)^2 n^{-1/2} (\log^{-2} n) (\log \log n) ) = o( (d^{i/2}/n)^2 q^{16} )
$$ 
fraction of four-tuples is considered here, and the contribution of this case to $X^2$ is negligible.

\bigskip \noindent \textbf{Case 8:} $|\{u', v'\} \cap U| = 2$, and $|\{x',y'\} \cap U|=2$. 


First, let us observe that in order to have $X_{u',v',x',y'} = 1$ the vertices $u',v',x',y'$ have to lie on an induced cycle of length $2i$ (of course, this is a necessary condition only). Moreover, note that there is only one such cycle in the graph induced by $U$ (namely, the one going through $u,v,x,y$). Hence, in order for this necessary condition to hold ``for free,'' all four vertices of $u',v',x',y'$ have to be on this cycle.  Thus, there are only $(2i)^4=O( (\log n / \log \log n)^4 )=o(\E{X})$ such $4$-tuples of vertices, and this contribution is negligible. 
Otherwise, we observe that at least one edge of the cycle $u',v',x',y'$ is not yet present in the graph induced by $U$, say an edge on the path between $u'$ and $x'$. For this to happen, there must be an edge between $S(u',k)$ and $S(x',i/2-1-k)$ for some $k \in \{0,1, \ldots,i/2-1\}$ that is not yet exposed. By the same calculations as in Case 3, this happens with probability at most
$$
O \left( \frac {d^{i/2}}{n} \cdot \frac {\log n}{\log \log n} \right).
$$
Thus,
$$
  \Prob{(X_{u',v',x',y'}=1 |  X_{u,v,x,y}=1)} = O \left( \Prob{(X_{u,v,x,y}=1)} q^{-16} \left( \frac{n}{d^{i/2}} \right)^3 (\log n / \log \log n) \right).
$$
On the other hand, since $u',v',x',y' \in U$, only a 
$$
O((|U|/n)^4) = O( (d^{i/2-1}/n)^3 n^{-1/2} (\log^{-2} n) (\log \log n) ) = o( (d^{i/2}/n)^3 q^{16} (\log \log n/\log n))
$$ 
fraction of four-tuples is considered here, and the contribution of this case to $X^2$ is negligible. 
\end{proof}

Using Lemma~\ref{lem:firstmoment} and Lemma~\ref{lem:secondmoment},  Theorem~\ref{thm:main2} follows now easily by Chebyshev's inequality.
\section{Concluding remarks and open questions} We have shown that in $\mathcal{G}(n,p)$, for $p \gg \frac{\log^5 n}{n (\log \log n)^2}$ the hyperbolicity is (up to a possible difference of $1$) a monotone decreasing graph parameter.  In general, since trees as well as cliques have hyperbolicity $0$, the hyperbolicity is not monotone, but we conjecture that in $\mathcal{G}(n,p)$  with $p$ above the threshold of connectivity, the same behavior holds. Intuitively, if $p$ is close to the threshold of connectivity, then $\mathcal{G}(n,p)$ a.a.s.\ contains a lot of  long cycles, and there will not be many shortcuts, making the hyperbolicity of the graph large. After extending the definition of hyperbolicity to non-connected graphs by defining it as the maximum over all connected components, the situation is quite different. For $p < (1-\eps)/n$ for some $\eps > 0$, the hyperbolicity is $0$ a.a.s., but for $p > (1+\eps)/n$ for some $\eps>0$, the appearance of the giant component makes the hyperbolicity to tend to infinity a.a.s.\ (see also~\cite{sparse_graphs}). It would be interesting to investigate for which values of $p$ the hyperbolicity of $\mathcal{G}(n,p)$ is maximized, and what this value is. We also would like to know whether the hyperbolicity is monotone increasing up to its maximal value and then decreasing, or whether there  are several ``peaks.''

\section*{Acknowledgements} The first author would like to thank David Coudert and Guillaume Ducoffe for helpful discussions.

\end{document}